\documentclass[a4paper,11pt]{amsart}
\usepackage{amssymb}
\usepackage{amsmath}
\usepackage{amsthm}
\usepackage{graphicx}
\usepackage{color}
\usepackage{subcaption}
\newtheorem{thm}{Theorem}
\newtheorem*{maina}{Theorem A}
\newtheorem*{mainb}{Theorem B}

\newtheorem{corollary}[thm]{Corollary}

\newtheorem{lemma}[thm]{Lemma}
\newtheorem{prop}[thm]{Proposition}

\newtheorem{problem 1}{Problem 1}
\newtheorem{problem 2}{Problem 2}
\newtheorem{problem 3}{Problem 3}
\theoremstyle{definition}
\newtheorem{question}[thm]{Question}

\newtheorem{remark}[thm]{Remark}

\title[Zeros of the Ising model]{Location of zeros for the partition function of the Ising model on bounded degree graphs}

\author[H. Peters]{Han Peters}
\address{H. Peters: Korteweg de Vries Institute for Mathematics, University of Amsterdam, Science Park 107, 1090GE Amsterdam, the Netherlands}
\email{hanpeters77@gmail.com}

\author[G. Regts]{Guus Regts$^{\dag}$}
\address{G. Regts: Korteweg de Vries Institute for Mathematics, University of Amsterdam, Science Park 107, 1090GE Amsterdam, the Netherlands.}
\email{guusregts@gmail.com}

\thanks{$^{\dag}$ Supported by a personal NWO Veni grant.}

\numberwithin{equation}{section}
\begin{document}

\begin{abstract}
The seminal Lee-Yang theorem states that for any graph the zeros of the partition function of the ferromagnetic Ising model lie on the unit circle in $\mathbb C$. In fact the union of the zeros of all graphs is dense on the unit circle.
In this paper we study the location of the zeros for the class of graphs of bounded maximum degree $d\geq 3$, both in the ferromagnetic and the anti-ferromagnetic case.
We determine the location exactly as a function of the inverse temperature and the degree $d$.
An important step in our approach is to translate to the setting of complex dynamics and analyze a dynamical system that is naturally associated to the partition function.

\begin{footnotesize}
\end{footnotesize}
\end{abstract}

\maketitle

%\footnote{Korteweg de Vries Institute for Mathematics, University of Amsterdam. Email: \texttt{guusregts@gmail.com}. Second author supported by a personal NWO Veni grant.}

\section{Introduction and main result}
For a graph $G=(V,E)$, $\xi,b\in \mathbb C$, the \emph{partition function of the Ising model} $Z_G(\xi,b)$ is defined as
\begin{equation}\label{eq:def pf}
Z_G(\xi,b):=\sum_{U\subseteq V}\xi^{|U|} \cdot b^{|\delta(U)|},
\end{equation}
where $\delta(U)$ denotes the collection of edges with one endpoint in $U$ and one endpoint in $V\setminus U$.
If $\xi$ and $b$ are clear from the context, we will often just write $Z_G$ instead of $Z_G(\xi,b)$.
In this paper we typically fix $b$ and think of $Z_G$ as a polynomial in $\xi$.
The case $0< b<1$ is often referred to as the ferromagnetic case, while the case where $b>1$ is referred to as the anti-ferromagnetic case.
%The parameter $b$ is related to the inverse temperature.

The Ising model is a simple model to study ferromagnetism in statistical physics.
In statistical physics the partition function of the Ising model is often written as
\begin{equation}\label{eq:def pf phys}
\sum_{\sigma:V\to \{-1,1\}}\exp\left (J/T\cdot \sum_{\{u,v\}\in E}\sigma(u)\sigma(v)+h/T\cdot \sum_{v\in V}\sigma(v)\right ),
\end{equation}
where $J$ denotes the coupling constant, $h$ the external magnetic field and $T>0$ the temperature, normalizing the Boltzmann constant to $1$ for convenience.
Setting $\xi=e^{-2h/T}$, $b=e^{-2J/T}$ and $U = \{v: \sigma(v) = -1\}$, then up to a factor of $e^{J\cdot|E|/T+h\cdot |V|/T}$, the two partition functions \eqref{eq:def pf} and \eqref{eq:def pf phys} are the same.

Lee and Yang~\cite{LY52II} proved that for any graph $G$ and any $b\in [-1,1]$ all zeros of $Z_G$ lie on the unit circle in $\mathbb C$. Their result attracted enormous attention in the literature, and similar statements have been proved in much more general settings, see for example~\cite{LS, Newman, Ru71,SF71,MS96,BM97,BG01,BB09,BB09a,Ru10,LSS17,Roederetal18,LSS18}.

In both the ferromagnetic and the antiferromagnetic case the union of the roots of $Z_G$ over all graphs $G$ lies dense in the unit circle. Density in the ferromagnetic case in fact follows from our results, as will be pointed out in Remark \ref{rem:dens}. It is natural to wonder for which classes of graphs and choice of parameters $b$ there are zero-free regions on the circle.
For the class of binary Cayley trees (see Section~\ref{sec:ratios} for a definition) this question has been studied by Barata and Marchetti~\cite{BM97} and Barata and Goldbaum~\cite{BG01}.
In the present paper we focus on the collection of graphs of bounded degree, and completely describe the location of the zeros for this class of graphs.
For $d\in \mathbb{N}$ we denote by $\mathcal{G}_d$ the collection of graphs with maximum degree at most $d$. By $\mathbb{D}$ we denote the open unit disk in $\mathbb{C}$.
Moreover, we will occasionally abuse notation and identify $(-\pi,\pi]$ with $\partial \mathbb{D}$, the unit circle. Given $\theta \in (-\pi, \pi)$ we write
$$
I(\theta):=\{e^{i\vartheta}\mid \vartheta\in (-\theta,\theta)\}.
$$
Our main results are:

\begin{maina}[ferromagnetic case]\label{thm:main}
Let $d\in \mathbb{N}_{\geq 2}$ and let $b\in (\frac{d-1}{d+1}, 1)$. Then there exists $\theta_b\in (-\pi,\pi)$ such that the following holds:
\begin{itemize}
\item[(i)]  for any $\xi\in I(\theta_b)$ and any graph $G\in \mathcal{G}_{d+1}$ we have $Z_G(\xi,b)\neq 0$;
\item[(ii)] the set $\{\xi=e^{i\vartheta}\in \partial \mathbb{D}\setminus  I(\theta_b)\mid Z_G(\xi,b)=0 \text{ for some }G\in \mathcal{G}_{d+1}\}$ is dense in $\partial \mathbb{D}\setminus  I(\theta_b)$.
\end{itemize}
\end{maina}
The dependence of $\theta_b$ on $b$ is given explicitly in (the proof of) Lemma~\ref{lem:derivative 1}.
For now we remark that as $b\to \frac{d-1}{d+1}$, $\theta_b\to 0$ and as $b\to 1$, $\theta_b\to \pi$.

We remark that part (ii) has recently been independently proved by Chio, He, Ji, and Roeder~\cite{Roederetal18}.
They focus on the class of Cayley trees and obtain a precise description of the limiting behaviour of the zeros of the partition function of the Ising model.

We recall that in the anti-ferromagnetic case the parameters $\xi$ for which $Z_G(\xi,b) = 0$ do not need to lie on the unit circle. For temperatures above the critical temperature, which corresponds to $b \in (1, \frac{d+1}{d-1})$ in our setting, the existence of a zero-free disk normal to the unit circle containing the point $\xi = +1$ was proved by Lieb and Ruelle \cite{LR}. Here we describe the maximal disk that can be obtained:

\begin{mainb}[anti-ferromagnetic case]\label{thm:main2}
Let $d\in \mathbb{N}_{\geq 2}$ and let $b\in (1, \frac{d+1}{d-1})$. Then there exists $\alpha_b\in (-\pi,\pi)$ such that the following holds:
\begin{itemize}
\item[(i)]  for any $\xi\in I(\alpha_b)$, any $r\geq 0$ and any graph $G\in \mathcal{G}_{d+1}$ we have $Z_G(r\cdot \xi,b)\neq 0$;
\item[(ii)] the set $\{\xi\in \mathbb{C} \mid Z_G(\xi,b)=0 \text{ for some }G\in \mathcal{G}_{d+1}\}$ accumulates on $e^{i\alpha_b}$ and $e^{-i\alpha_b}$.
\end{itemize}
\end{mainb}
The value of $\alpha_b$ can again be explicitly expressed in terms of $b$, see Figure \ref{fig:lambda01} for an illustration depicting $\alpha_b$.% and $\theta_b$ in the anti-ferromagnetic case.

%It will be clear from our proof that statement (i) holds also for $\xi$ in the $\mathbb R_{\ge 0}$-cone through $I(\vartheta)$.

Another recent related contribution to the Lee-Yang program is due to Liu, Sinclair and Srivastava~\cite{LSS18}, who showed that for $\xi=1$ and $d\geq 2$ there exists an open set $B\subset \mathbb C$ containing the interval $(\frac{d-1}{d+1},\frac{d+1}{d-1})$ such that for any $G\in \mathcal G_{d+1}$ and $b\in B$, $Z_G(1,b)\neq 0$.

\subsection{Motivation}
The motivation for studying the location of zeros of partition functions traditionally comes from statistical physics.
Since this is well known and since many excellent expositions exist, see for example~\cite[Section 7.4]{Ba17}, we choose not discuss the physical background here.
However, recently there has also been interest in understanding the location of zeros from the perspective of theoretical computer science, more precisely from the field of approximate counting.

In theoretical computer science it is known that the exact computation of partition functions, such as of the Ising model and the hardcore model, or the number of proper $k$-colorings of a graph $G$, generally is a \#P-hard problem (i.e., it is as hard as computing the number of Hamiltonian cycles in a graph, see~\cite{V79a,V79b,AB09} for detailed information on the class \#P).
For this reason much effort has been put in designing efficient approximation algorithms.
Traditionally such algorithms are randomized and are based on Markov chains, see \cite{J03}.
In particular, Jerrum and Sinclair \cite{JS93} showed that for all $0<b<1$ and $\xi>0$ the partition function of the Ising model can be efficiently approximated on any graph $G$.
Another approach is based on decay of correlations and was initiated by Weitz~\cite{W06}. This leads to deterministic approximation algorithms.
Using decay of correlations, Sinclair, Srivastava and Thurley \cite{SST14} gave an efficient deterministic approximation algorithm for computing the Ising partition function on graphs of maximum degree at most $d+1$ $(d\geq 2)$ when $\xi=1$ and  $b \in (\frac{d-1}{d+1},1]$.

Recently a new approach for obtaining deterministic approximation algorithms was proposed by Barvinok, see \cite{Ba17}, based on truncating the Taylor series of the logarithm of the partition functions in regions where the partition function is nonzero.
It was shown by Patel and the second author in \cite{PaR17} that this approach in fact yields polynomial time approximation algorithms when restricted to bounded degree graphs.
Combining the approach from~\cite{PaR17} (cf. \cite{LSS17}) with Theorem A and the original Lee-Yang result, we immediately obtain the following as a direct corollary:
\begin{corollary}
Let $d\geq 2$, let $b \in (\frac{d-1}{d+1},1]$ and let $\xi \in I(\theta_b)$, for $\theta_b$ as in Theorem A.
Then for any $\varepsilon>0$ there exists an algorithm that, given an $n$-vertex graph $G$ of maximum degree at most $d+1$, computes a relative $\varepsilon$-approximation\footnote{A relative $\epsilon$-approximation to a nonzero complex number $x=e^{a}$ is a nonzero complex number $z=e^{b}$ such that $|a-b|<\varepsilon$.} to $Z_G(\xi,b)$ in time polynomial in $n/\varepsilon$.
\end{corollary}
An identical statement holds for $b>1$, except there it does not follow directly from Theorem B. One also needs that for $\xi$ in a small disk around zero the partition function does not vanish, see Remark~\ref{rem:zero}.

Given the recent progress on understanding the complexity of approximating independence polynomial at nonpositive fugacities based on connections to complex dynamics due to Bezakov\'a, Galanis, Goldberg and \v{S}tefankovi\v{c}~\cite{BGGS2017}, a natural question that arises is the following:
\begin{question}
Let $d\geq 2$. Is it NP-hard (or maybe even \#P-hard) to approximate the partition function of the Ising model  on graphs of maximum degree at most $d+1$ when $b \in (\frac{d-1}{d+1},1)$ and $\xi\in \partial \mathbb D\setminus I(\theta)$?
In fact even the hardness of approximating the partition function of the Ising model for  $\xi\in \partial \mathbb D$ and $0\leq b\leq \frac{d-1}{d+1}$ is not known. See~\cite{GG17} for some related hardness results.
\end{question}

\subsection{Approach}
Our approach to proving our main theorems is to make use of the theory of complex dynamics and combine this with some ideas from the approximate counting literature. The value of the partition function of a Cayley tree can be expressed in terms of the value of the partition function for the Cayley tree with one fewer level, inducing the iteration of a univariate rational function. Understanding the dynamical behaviour of this function leads to understanding of the location of the zeros of the partition function for Cayley trees. The same approach forms the basis of \cite{BM97,BG01,Roederetal18,LSS18}.
To prove our result for general bounded degree graphs, we use the tree of self avoiding walks, as defined by Weitz~\cite{W06}, to relate the partition function of a graph to the partition function of a tree with additional boundary conditions.
This relationship no longer gives rise to the iteration of a univariate rational function, but with some additional effort we can still transfer the results for the univariate case to this setting.

We remark that a similar approach was used by the authors in~\cite{PR17} to answer a question of Sokal concerning the location of zeros of the independence polynomial, a.k.a., the partition function of the hard-core model.

Complex dynamics has also been used to study the location of zeros the chormatic polynomials of certain trees by Royle and Sokal. See the appendix of the arxiv version of \cite{RS15}.

\subsubsection{Organization}
This paper is organized as follows.
In the next section we will define ratios of partition functions, and prove that for Cayley trees this gives rise to the iteration of a univariate rational function. In Section~\ref{sec:complex dynamics} we employ basic tools from complex dynamics to analyze this iteration.
In particular a proof of part (ii) of our main theorems will be given there.
Finally, in Section~\ref{sec:zero-free} we collect some additional ideas and provide a proof of part (i) of our main theorems.

\section{Ratios}\label{sec:ratios}
At a later stage, in Section~\ref{sec:zero-free}, it will be convenient to have a multivariate version of the Ising partition function defined for a graph $G=(V,E)$, complex numbers $(\xi_v)_{v\in V}$, and $b\in \mathbb{C}$ as follows:
\[
Z_G((\xi_v),b):=\sum_{U\subseteq V}\prod_{u\in U}\xi_u \cdot b^{|\delta(U)|}.
\]
The two-variable version is obtained from this version by setting all $\xi_v$ equal.
We will often abuse notation and just write $Z_G(\xi,b)$ for the multivariate version.

Let $G=(V,E)$ be a graph and let $X\subseteq V$.
We call any map $\tau:X\to \{0,1\}$ a \emph{boundary condition on $X$}.
Let now $\tau$ be a boundary condition on $X\subseteq V$. We say that $U\subseteq V$ is compatible with $\tau$ if for each vertex $u\in X$ with $\tau(u)=1$ we have $u\in U$ and for each vertex $u\in X$ with $\tau(u)=0$ we have $u\notin U$. We shall write $U\sim \tau$ if $U$ is compatible with $\tau$.
We define
\[
Z_{G,\tau}(\xi,b)=\sum_{\substack{U\subseteq V\\U\sim \tau}}\prod_{u\in U}\xi_u \cdot b^{|\delta(U)|}.
\]
Fix a vertex $v$ of $G$. We let $\tau_{v,0}$, and $\tau_{v,1}$ respectively, denote the boundary conditions on $X\cup \{v\}$ where $v$ is set to $0$, and $1$ respectively.
In case $v$ is contained in $X$, we consider $X\cup\{v\}$ as a multiset to make sure $\tau_{v,0}$ and $\tau_{v,1}$ are well defined.
For one element $\sigma\in \{\tau_{v,0},\tau_{v,1}\}$ the vertex $v$ gets two different values, in which case no set $U\subseteq V$ is compatible with $\sigma$ and consequently we set $Z_{G,\sigma}=0$.

We denote the extended complex plane, $\mathbb{C}\cup \{\infty\}$, by $\hat{\mathbb{C}}$.
We introduce the ratio  $R_{G,\tau,v}\in \hat{\mathbb {C}}$, by
\begin{equation}\label{eq:ratio}
R_{G,\tau,v}:=\left \{\begin{array}{ll}
\infty&\text{ if }Z_{G,\tau_{v,0}}(\xi,b)= 0,
\\
\frac{Z_{G,\tau_{v,1}}(\xi,b)}{Z_{G,\tau_{v,0}}(\xi,b)}& \text{ otherwise}.
\end{array}\right.
\end{equation}
We remark that $R_{G,\tau,v}$ equals a rational function in $\xi$, except perhaps for values of $\xi$ for which $Z_{G,\tau_{v,0}}$ and $Z_{G,\tau_{v,1}}$ vanish simultaneously. We will however prove in Lemma \ref{lemma27} that this can never happen for the $\xi$ we care about, and therefore it is safe to think of $R$ as a rational function in $\xi$.

If no boundary condition is present, or if it is clear from the context, we just write $R_{G,v}$ for the ratio.
We have the following trivial, but important, observation:
\begin{align}\label{eq:important observation}
\text{ if }&Z_{G,\tau_{v,0}}\neq 0 \text{, or }Z_{G,\tau_{v,1}}\neq 0 \text{, then:}\nonumber
\\
&R_{G,\tau,v}\neq -1 \text{ if and only if } Z_{G,\tau}\neq 0.
\end{align}

The following lemma shows how to express the ratio for trees in terms of ratios of smaller trees.
\begin{lemma}\label{lem:ratio}
Let $G=(V,E)$ be a tree with boundary condition $\tau$ on $X\subseteq V$.
Let $u_1,\ldots,u_d$ be the neighbors of $v$ in $G$, and let $G_1,\ldots,G_d$ be the components of $G-v$ containing $u_1,\ldots,u_d$ respectively.
We just write $\tau$ for the restriction of $\tau$ to $X\cap V(G_i)$ for each $i$.
For $i=1,\ldots, d$ let $\tau_{i,0}$ and $\tau_{i,1}$ denote the respective boundary conditions obtained from $\tau$ on $(X\cup\{u_i\})\cap V(G_i)$ where $u_i$ is set to $0$ and $1$ respectively.
%If $Z_{G_i,u_i}$ is well defined for $i=1,\ldots,d$, then
If for each $i$, not both $Z_{G_i,\tau_{i,0}}(\xi,b)$ and $Z_{G,\tau_{i,1}}(\xi,b)$ are zero, then
\begin{equation}
R_{G,\tau,v}=\xi_v\prod_{i=1}^d\frac{R_{G_i,\tau, u_i}+b}{b R_{G_i,\tau, u_i}+1}.
\end{equation}
\end{lemma}
\begin{proof}
We can write
\[
Z_{G,\tau_{v,1}}=\xi_v \prod_{i=1}^d\left(Z_{G_i,\tau_{i,1}}+b Z_{G_i,\tau_{i,0}} \right),
\text{ and }
Z_{G,\tau_{v,0}}=\prod_{i=1}^d\left(b Z_{G_i,\tau_{i,1}}+Z_{G_i,\tau_{i,0}} \right).\]
Let us fix $i\in \{1,\ldots,d\}$.
%Since $R_{G_i,u_i}=\frac{Z_{G_i,\tau_{i,1}}}{Z_{G_i,\tau_{i,0}}}$ is well defined, we have that not both $Z_{G_i,\tau_{i,1}}$ and $Z_{G_i,\tau_{i,0}}$ are zero.
Suppose first that $Z_{G_i,\tau_{i,0}}\neq 0$. Then we can divide the numerator and denominator by $Z_{G_i,\tau_{i,0}}$ to obtain
\begin{equation}\label{eq:ratio i}
\frac{Z_{G_i,\tau_{i,1}}+b Z_{G_i,\tau_{i,0}}}{b Z_{G_i,\tau_{i,1}}+Z_{G_i,\tau_{i,0}}}=\frac{R_{G_i,\tau,u_i}+b}{b R_{G_i,\tau,u_i}+1}.
\end{equation}
If $Z_{G_i,\tau_{i,0}}=0$, then on the left-hand side of \eqref{eq:ratio i} we obtain $1/b$ while on the right-hand side, plugging in $R_{G_i,\tau,u_i}=\infty$, we also obtain $1/b$.
Therefore this expression is also valid when $Z_{G_i,\tau_{i,0}}=0$.
This finishes the proof.
\end{proof}

Let us now specialize the previous lemma to a special class of (rooted) trees.
Fix $d\in \mathbb{N}_{\geq 2}$.
The tree $T_{0,d}$ consists of single vertex, its root. For $k\geq 1$, the tree $T_{k,d}$ consists of a root vertex $v$ of degree $d$ with each edge incident to $v$ connected to the root of a copy of $T_{k-1,d}$.
This class of trees is also known as the class of (rooted) \emph{Cayley trees}.
If $d$ is clear from the context, we just write $T_k$ instead of $T_{k,d}$.

Define $f=f_{\xi,b}:\hat{\mathbb C} \to \hat{\mathbb C}$ for $R\in \hat{\mathbb C}$ by
\begin{equation}\label{eq:define f}
f(R):= \xi \cdot \left(\frac{R + b}{b R + 1}\right)^d.
\end{equation}
Let us moreover, define $g:\hat{ \mathbb{C}}\to \hat{\mathbb  C}$ by $g: R\mapsto \frac{R+b}{b R+1}$ for $R\in \hat{ \mathbb C}$ so that $f(R)=\xi g(R)^d$.
Since $b$ is real, it follows that the M\"obius transformation $g$ preserves $\partial \mathbb{D}$.
If $\xi\in \partial\mathbb{D}$ the same holds for $f$.
\begin{corollary}\label{cor:tree to dynamics}
Let $d\in \mathbb{N}_{\geq 2}$ and let $\xi\in \partial \mathbb{D}$ and $b\in \mathbb R$.
Then the orbit of $1$ under $f=f_{\xi,b}$ avoids $-1$ if and only if $Z_{T_k}\neq 0$ for all $k$.
\end{corollary}
\begin{proof}
We note that $f(1)=\xi$, hence we may just as well consider the orbit of $\xi$.
We observe that, as there is no boundary condition, $R_{T_{0},v}=\xi$.
Now suppose that $Z_{T_k}\neq 0$ for all $k\in \mathbb{N}$. Then by \eqref{eq:important observation} we see that $R_{T_{k},v}\neq -1$ for all $k$.
Then since $Z_{T_k,\tau_{v,0}}+Z_{T_k,\tau_{v,1}}=Z_{T_k}\neq 0$, it follows that either $Z_{T_k,\tau_{v,0}}\neq 0$, or $Z_{T_k,\tau_{v,1}}\neq 0$.
By Lemma~\ref{lem:ratio} we obtain
\begin{equation}\label{eq:iteration}
R_{T_{k},v}=f(R_{T_{k-1},v})=\ldots= f^{\circ k}(\xi)
\end{equation}
and hence the orbit of $1$ avoids $-1$.

Conversely, suppose that the orbit of $1$ avoids $-1$, while $Z_{T_{k}}=0$ for some $k$.
Then let $k$ be the smallest integer for which $Z_{T_{k}}=0$. Then we have $Z_{T_k,\tau_{v,0}}\neq 0$.
Indeed, $Z_{T_k,\tau_{v,0}}=(Z_{T_{k-1},\tau_{v,0}}+bZ_{T_{k-1},\tau_{v,1}})^d$ and by assumption, as $Z_{T_{k-1}}\neq 0$ we have that one of $Z_{T_{k-1},\tau_{v,0}}$ and $Z_{T_{k-1},\tau_{v,1}}$ is nonzero. But since
\[
\frac{Z_{T_{k-1},\tau_{v,0}}}{Z_{T_{k-1},\tau_{v,1}}}=R_{T_{k-1},v}=f^{\circ k-1}(\xi)\in \partial \mathbb D\setminus\{-1\},
\]
and therefore $Z_{T_{k-1},\tau_{v,0}}+b{Z_{T_{k-1},\tau_{v,1}}}\neq 0$, as desired.
Now \eqref{eq:important observation} implies that $R_{T_{k}}=-1$: a contradiction.
This finishes the proof.
\end{proof}

This corollary motivates the study of the complex dynamical behaviour of the map $f$ at starting point $1$ (or $\xi$).
We will do this in the next section, returning to general graphs in Section~\ref{sec:zero-free}

\section{Complex dynamics of the map $f_{\xi,b}$}\label{sec:complex dynamics}
Let $d\in \mathbb{N}_{\geq 2}$ and let $b\in \mathbb{R}$.
In this section we study the dynamical behavior of the map $f_{\xi,b}$ for $\xi$ of norm $1$.
It is our aim to prove the following results.

\begin{thm}[Ferromagnetic case]\label{thm:dynamicsF}
Let $d\in \mathbb{N}_{\geq 2}$ and let $b\in (\frac{d-1}{d+1},1)$. There exists $\theta_b\in (0,\pi)$ such that
\begin{itemize}
\item[(i)] for each $\xi=e^{i\vartheta}$ with $\vartheta\in (-\theta_b,\theta_b)$ there exists a closed circular interval $I_b\subset \partial \mathbb{D}$, with $1$ as boundary point, which is forward invariant under $f_{\xi,b}$ and does not contain $-1$.
In particular, the orbit of $R=\xi$ under $f_{\xi,b}$ avoids the point $-1$;
\item[(ii)] The interval $(-\theta_b,\theta_b)$ is maximal: The collection $\{\xi\}\subset \partial \mathbb{D}$ for which the orbit of $R=\xi$ under $f_{\xi,b}$ lands on $-1$ is dense in $\partial \mathbb{D}\setminus(-\theta_b,\theta_b)$.
\end{itemize}
\end{thm}

\begin{remark}
We can provide an explicit formula for $\theta_b$ as a function of $b$, see Lemma~\ref{lem:derivative 1} and its proof below.
\end{remark}
While a variant of this result was also independently proved in \cite{Roederetal18} we will provide a proof for it, as certain parts and ideas of our proof will be used to prove the next theorem.%in Section~\ref{sec:zero-free}.

\begin{thm}[Anti-ferromagnetic case]\label{thm:dynamicsAF}
Let $d\in \mathbb{N}_{\geq 2}$ and let $b\in (1,\frac{d+1}{d-1})$. There exists $\alpha=\alpha_b\in (0,\pi)$ such that
\begin{itemize}
\item[(i)] for each $\xi=e^{i\vartheta}$ with $\vartheta\in (-\alpha,\alpha)$ the shortest closed circular interval with boundary points $1$ and $\xi$, $I_b$, is forward invariant under $f_{\xi,b}$.
In particular, the orbit of $R=\xi$ under $f_{\xi,b}$ avoids the point $-1$;
\item[(ii)] The interval $(-\alpha,\alpha)$ is maximal: The collection $\{\xi\}\subset \mathbb{C}$, for which the orbit of $R=\xi$ under $f_{\xi,b}$ lands on $-1$ accumulates on $e^{\pm i \alpha}$.
\end{itemize}
\end{thm}

Observe that by Corollary~\ref{cor:tree to dynamics}, part (ii) of these theorems implies part (ii) of our main theorems.
Part (i) of our main theorems, which will be proved in Section~\ref{sec:zero-free}, will rely upon parts (i) in Theorem~\ref{thm:dynamicsF} and Theorem~\ref{thm:dynamicsAF}.

We moreover note that while both theorems look quite similar, they are not quite the same.
In particular, it is not clear whether in the anti-ferromagnetic case roots lie dense on the circular arc containing $-1$ between $\alpha_b$ and $-\alpha_b$.
This question has been studied in recent follow up work of Bencs, Buys, Guerini and the first author~\cite{BBGP19}.

The difference in nature is also apparent in the proofs of these results.
To prove these results, we start with some observations from (complex) analysis and complex dynamics concerning the map $f_{\xi,b}$, after which we first prove Theorem~\ref{thm:dynamicsF} and then  Theorem~\ref{thm:dynamicsAF}.

\subsection{Observations from analysis and complex dynamics}
\subsubsection{Elementary properties of $f_{\xi,b}$}
We start with some basic complex analytic properties of the map $f_{\xi,b}$.
Throughout we assume that $b$ is real valued, $\xi\in \partial \mathbb{D}$ and we write $f=f_{\xi,b}$.
We first of all note that if $b=1$, the map $f$ just equals multiplication by $\xi$. Therefore we will restrict to $b\neq 1$.

The behavior of $f$ on the outer disk $\hat {\mathbb C} \setminus \overline{\mathbb D}$ is conjugate to that on the inner disk $\mathbb D$:

\begin{lemma}\label{lemma:anti}
The map $f$ is invariant under conjugation by the the anti-holomorphic map
$$
R \mapsto \frac{1}{\overline{R}}.
$$
\end{lemma}
\begin{proof}
First of all, we have $g(1/R)=1/g(R)$.
Now since $\overline \xi=1/\xi$, it follows that $f(1/\overline{R})=1/\overline{f(R)}$, as desired.
\end{proof}
Thus, for most purposes it is sufficient to consider only the behavior on $\mathbb D$ and on $\partial \mathbb D$.

\begin{lemma}\label{lemma:covering}
Let $\xi\in \partial\mathbb{D}$. For $b\in \mathbb R\setminus \{1\}$ the map $f_{\xi,b}$ induces a $d$-fold covering on $\partial \mathbb{D}$.
For $|b| < 1$ this covering is orientation preserving, for $|b|>1$ it is orientation reversing.
\end{lemma}
\begin{proof}
Since $f=f_{\xi,b}$ has no critical points on $\partial \mathbb{D}$ it follows that the map is a $d$-fold covering for any real $b$.
If $-1<b<1$ both $\mathbb{D}$ and $\hat{\mathbb C} \setminus \mathbb{D}$ are invariant under $f$, hence conformality of $f$ near $\partial \mathbb D$ implies that $f$ is orientation preserving.
If $|b|>1$ then $f$ maps $\mathbb{D}$ into ${\mathbb C} \setminus \mathbb{D}$ and vice versa, which implies that $f$ is orientation reversing.
\end{proof}

From now on we will only consider $b>0$. The derivative of $f$ satisfies:
\begin{equation}\label{eq:der f}
f^\prime(R)=\xi d\left( \frac{R+b}{b R+1}\right)^{d-1}\cdot \frac{1-b^2}{(b R+1)^2}=f(R)\frac{d(1-b^2)}{(R+b)(b R+1)}.
\end{equation}
It follows that $|f^\prime(R)|$ is independent of $\xi$ and, since $b>0$, is strictly increasing with $|\mathrm{Arg}(R)|$.

Let us define \[b_c:=\frac{d-1}{d+1}.\]
Note that
$$
|f_{1,b}^\prime(1)| = d\cdot \frac{1-b}{1+b},
$$
from which it follows that $|f^\prime_{\xi,b}(1)|> 1$ when $0< b< b_c$ or $b> \frac{1}{b_c}$, $|f^\prime_{\xi,b}(1)| = 1$ when $b =  b_c$ or $b = \frac{1}{b_c}$, and $|f^\prime_{\xi,b}(1)| < 1$ when $b_c < b<1$ or $1<b<1/b_c$.

Recall that a map is said to be \emph{expanding} if it locally increases distances, and \emph{uniformly expanding} if distances are locally increased by a multiplicative factor bounded from below by a constant strictly larger than $1$. Our above discussion implies the following.

\begin{lemma}
Let $\xi\in \partial \mathbb{D}$.
\begin{enumerate}
\item[(i)] If $0<b < b_c$ or $b>1/b_c$, then  the covering $f_{\xi,b}|_{\partial \mathbb{D}}$ is uniformly expanding.
\item[(ii)] If $b = b_c$, or if $b=1/b_c$, then the covering $f_{\xi,b}|_{\partial \mathbb{D}}$ is expanding, but not uniformly expanding: $|f_{\xi,b}^\prime(1)| = 1$.
\item[(iii)] If $b_c < b < 1/b_c$, then $|f_{\xi,b}^\prime(1)| < 1$.
\end{enumerate}
\end{lemma}

\begin{lemma}\label{lem:der real}
Let $b\in \mathbb R$ and $\xi\in \partial \mathbb{D}$.
Let $R_0\in \partial\mathbb{D}$ be a fixed point of $f=f_{\xi,b}$.
Then $f^\prime(R_0)\in \mathbb{R}$.
\end{lemma}
\begin{proof}
Let us denote the tangent space at the circle of a point $R$ by $T_ R$; this is spanned by some vector in $\mathbb{C}=\mathbb{R}^2$.
Then since the derivative is a linear map from $T_{R_0}$ to $T_{f(R_0)}=T_{R_0}$, it follows that $f'(R_0)$ has to be a real number.
\end{proof}

\subsubsection{Observations from complex dynamics}
We refer to the book ~\cite{Mil} for all necessary background. Throughout we will assume that $b >0$, $b \neq 1$, $\xi\in \partial \mathbb{D}$ and we write $f=f_{\xi,b}$.

By Montel's Theorem the family of iterates $\{f^{\circ n}\}$ is normal on $\mathbb{D}$ and on $\hat {\mathbb C} \setminus \overline{\mathbb D}$.
Recall that the set where the family of iterates is locally normal is called the \emph{Fatou set}, and its complement is the \emph{Julia set}. Thus, the Julia set of $f$ is contained in $\partial \mathbb D$, and there are two possibilities for the connected components of the Fatou set, i.e. the Fatou components:

\begin{lemma}
Either the Fatou set of $f$ consists of precisely two Fatou components, $\mathbb D$ and $\hat {\mathbb C} \setminus \overline{\mathbb D}$, or there is only a single Fatou component which contains both $\mathbb D$ and $\hat {\mathbb C} \setminus \overline{\mathbb D}$. In the latter case the component is necessarily invariant. In the former case the two components are invariant when $b < 1$, and are periodic of order $2$ when $b > 1$.
\end{lemma}

Recall that invariant Fatou components are classified: each invariant Fatou component is either the basin of an attracting or parabolic fixed point, or a rotation domain. An invariant attracting or parabolic basin always contains a critical point, while a rotation domain does not. The critical points of $f$ are $-b, -1/b$, hence it follows that in both of the above cases the Fatou components must be either parabolic or attracting.

If there is only one Fatou component, by Lemma \ref{lemma:anti} this component must be an attracting or parabolic basin of a fixed point lying in $\partial \mathbb D$. If there are two Fatou components then they are either both attracting basins, or they are both basins of a single parabolic fixed point in $\partial \mathbb D$. We emphasize that there can be no other parabolic or attracting cycles.

The parameters for which there exist parabolic fixed points will play a central role in our analysis.

\begin{lemma}\label{lem:derivative 1}
Let $b \in (b_c, 1) \cup (1, \frac{1}{b_c})$. Then there exists a unique $\theta=\theta_b\in (0,\pi)$ such that for $\xi=e^{\pm i\theta} \in \partial \mathbb D$ the function $f_{\xi,b}$ has a (unique) parabolic fixed point. Moreover the following holds:
\begin{itemize}
\item[(i)] If $b<1$, then the parabolic fixed point $R$ of $f=f_{\xi,b}$ satisfies $f^\prime(R) = 1$ and is a solution of the equation
\begin{equation}
R^2 + \frac{d(b^2-1)+(1+b^2)}{b} R + 1 = 0, \label{eq:beta<1}
\end{equation}
\item[(ii)]
If $b>1$, then the parabolic fixed point $R$ of $f=f_{\xi,b}$ satisfies $f^\prime(R)=-1$ and is a solution of the equation
\begin{equation}
R^2 + \frac{d(1-b^2)+(1+b^2)}{b} R + 1 = 0.\label{eq:beta>1}
\end{equation}
\end{itemize}
\end{lemma}
\begin{proof}
Recall that for fixed $b$ the value of $|f^\prime(R)|$ is independent of $\xi$, depends only on $|\mathrm{Arg}(R)|$, is strictly increasing in $|\mathrm{Arg}(R)|$, and satisfies $|f^\prime(1)| < 1$ and $|f^\prime(-1)| > 1$. Thus there exists a unique pair of complex conjugates $R_0, \overline{R_0}$ for which $|f^\prime(R_0)| = |f^\prime(\overline{R_0})| = 1$. Hence there exists a unique $\xi_0$ for which $f_{\xi_0,b}(R_0) = R_0$, and by symmetry $f_{\overline{\xi_0},b}(\overline{R_0}) = \overline{R_0}$. Since the action of $f$ on the unit circle is orientation preserving for $b<1$, and orientation reversing for $b>1$, it follows by Lemma~\ref{lem:der real} that $f_{\xi_0,b}^\prime(R_0)$ equals $1$ for $b<1$, and equals $-1$ for $b>1$.

Let us first consider the case that $b<1$. We are then searching for solutions to the two equations
\begin{align}\label{equationone}
f(R) = \xi & \left(\frac{R + b}{1+b R}\right)^d  = R\\
\label{equationtwo} f^\prime(R) = \xi & d \left(\frac{R + b}{1+b R}\right)^{d-1} \frac{1-b^2}{(1+b R)^2} =  1.
\end{align}
Rewriting equation \eqref{equationone} gives
$$
\xi \left(\frac{R + b}{1+b R}\right)^{d-1}  = R \cdot \frac{1+b R}{R + b},
$$
which can be plugged into \eqref{equationtwo} to give
$$
dR\frac{1-b^2}{(1+b R)(R+b)} = 1,
$$
which is equivalent to
$$
R^2 + \frac{d(b^2-1)+(1+b^2)}{b} R + 1 = 0
$$
For $b_c < b < 1$ there are two solutions for $R$, a pair of complex conjugates lying on the unit circle. For each of these solutions there exists a unique value of $\xi \in \partial \mathbb D$ for which equation \eqref{equationone} is satisfied.
These values of $\xi$ are clearly complex conjugates of each other and, when the two solutions $R$ and $\overline{R}$ are distinct, must be distinct as $f_{\xi,b}$ has at most one parabolic fixed point.

If $b>1$, we need to replace $1$ by $-1$ on the right-hand side of \eqref{equationtwo}.
Similar to the $b<1$ case, this then leads to equation \eqref{eq:beta>1}, which, when $b<1/b_c$, has two solutions for $R$, a pair of complex conjugates lying on the unit circle.
As before, for each of these solutions there exists a unique value of $\xi \in \partial \mathbb D$ for which equation \eqref{equationone} is satisfied. Again these values of $\xi$ are complex conjugates of each other.
\end{proof}

We note that in the lemma above when $b = b_c$ or when $b=1/b_c$ there is a double solution at $R=1$, and hence the corresponding $\xi$ equals $1$.
For this map there are two separate parabolic basins: the inner and outer unit disk.
When $b_c < b < 1$ the parabolic fixed point is a double fixed point, and hence has only one parabolic basin. It follows that in this case there is a unique Fatou component, which contains both the inner and outer unit disk, and all orbits approach the parabolic fixed point along a direction tangent to the unit circle. When $1<b<1/b_c$ the inner and outer disk are inverted by $f$, the fact that $f^\prime = -1$ implies that orbits in these components converge to the parabolic fixed point along the direction normal to the unit circle, while nearby points on the unit circle move away from the parabolic fixed point.

We have now established some basic properties of the map $f$ and move on to the respective proofs of Theorems~~\ref{thm:dynamicsF} and ~\ref{thm:dynamicsAF}.

\subsection{Proof of Theorem~\ref{thm:dynamicsF}}
\subsubsection{Proof of part (i)}
We will consider the behavior for parameters $b< 1$ and $\xi\in \partial \mathbb{D}$ for which $f=f_{\xi,b}$ has an attracting fixed point on $\partial \mathbb D$.% and prove Theorem~\ref{thm:dynamics} (i) for these maps. We will consider the case $b>1$ later.

The Julia set $J$ of $f$, which is nonempty\footnote{In fact, it can be shown that the Julia set is a Cantor set.}, is contained in the unit circle, and the complement is the unique Fatou component, the (immediate) attracting basin.
The intersection of $\mathbb C \setminus J$ with the unit circle consists of countably many open intervals.
We refer to the interval containing the attracting fixed point as the \emph{immediate attracting interval}. We note that this interval is forward invariant, and the restriction of $f$ to this interval is injective.
We emphasize that we may indeed talk about \emph{the} immediate attracting interval, as there are no other parabolic or attracting cycles.

\begin{thm}\label{thm:im. at. interval}
Let $b\in (b_c,1)$ and let $\theta=\theta_b$ (from Lemma~\ref{lem:derivative 1}).
Then for $\xi\in \partial \mathbb D$ the map $f_{\xi,b}$ has an attracting or parabolic fixed point on $\partial \mathbb D$ if and only if $\xi\in \{e^{i\vartheta}\mid \vartheta\in [-\theta,\theta]\}$. If $\vartheta\in (-\theta,\theta)$, then the point $+1$ lies in the immediate attracting interval.
%, while $-1$ does not.
\end{thm}
\begin{proof}
We will consider the changing behavior of the map $f_{\xi,b}$ as $\xi \in \partial \mathbb D$ varies, for $b$ fixed. By the implicit function theorem the fixed points of $f_{\xi,b}$, i.e. the solutions of $f_{\xi, b}(R) - R = 0$, depend holomorphically on $\xi$, except when $f_{\xi,b}^\prime(R) = 1$.
By Lemma~\ref{lem:derivative 1} this occurs exactly at two parameters $\xi = e^{\pm i\theta}$.

Recall that the absolute value of the derivative, $|f_{\xi,b}^\prime(R)|$, is independent of $\xi$, strictly increasing in $|\mathrm{Arg}(R)|$, and that $|f_{\xi,b}^\prime(+1)| <1$ while $|f_{\xi,b}^\prime(-1)| >1$. For each $R\in \partial \mathbb D$ there exists a unique $\xi \in \partial \mathbb D$ for which $R$ is fixed, inducing a map $R \mapsto \xi(R)$, holomorphic in a neighborhood of $\partial \mathbb D$. Since there can be at most one attracting or parabolic fixed point on $\partial \mathbb D$, the map $R \mapsto \xi(R)$ is injective on the circular interval $\{ R \; : \; |f_{\xi,b}^\prime(R)| \le 1\}$. It follows that the image of this interval under $R \mapsto \xi(R)$ equals $\{e^{i\vartheta}\mid \vartheta\in [-\theta,\theta]\}$, and that for $\xi$ outside of this interval the function $f_{\xi,b}$ cannot have a parabolic or attracting fixed point on $\partial \mathbb D$.

When $f_{\xi,b}$ has an attracting fixed point on $\partial \mathbb D$, the boundary points of the immediate attracting interval are necessarily fixed points. The fact that there cannot be other attracting or parabolic cycles on $\partial \mathbb D$ implies that the two boundary points are repelling. It follows also that $R = +1$ cannot be a boundary point of the immediate attracting interval, and since these boundary points vary holomorphically (with $\xi$), and thus in particular continuously, it follows that $R= +1$ is always contained in the immediate attracting interval.

%It remains to show that $R=-1$ cannot lie in the immediate attracting interval, which follows from the fact that $-1$ cannot be one of the repelling boundary points of the immediate attracting interval.
%To see this, observe that $-1$ can only be fixed when $\xi = \pm 1$. depending on the parity of $d$. When $\xi = +1$ the point $-1$ does not lie in the immediate attracting basin since this interval is proper, and must be invariant under complex conjugation. When $\xi = -1$ the point $R = +1$ is mapped to $-1$, which is either fixed or maps back to $+1$. In either case $+1$ does not lie in the immediate attracting interval, hence there is no attracting fixed point. This completes the proof.
\end{proof}

To complete the proof of Theorem~\ref{thm:dynamicsF} (i) we need to define the circular interval $I_b$. We let $I_b$ be the shortest closed circular interval with boundary points $1$ and $R_0$, the attracting fixed point of $f$.
Then clearly $-1\notin I_b$.
Since $I_b$ is contained in the immediate attracting interval and since $f$ is orientation preserving it follows that that $I_b$ is forward invariant for $f$.

%We assume that $\pi>\phi_0\geq 0$, the other case being similar.
%Then, for all $z\in I_b$, the argument of $f(z)$ cannot be smaller than the argument of $z$ since $I_b$ is contained in the immediate attracting interval.
%On the other hand, we cannot have that the argument of $f(z)$ is bigger than that of $R_0$, as $f(R_0)=R_0$ and $f$ preserves orientation. So indeed $I_b$ is forward invariant under $f$.

This finishes the proof of part (i).

\subsubsection{Proof of part (ii)}
We start with analyzing what happens when $b\le b_c$.
\begin{prop}\label{expansion1}
For $0 < b \le b_c$ the parameters $\xi$ for which the orbit of $R_0 = \xi$ under the map $f_{\xi,b}$ takes on the value $-1$ is dense in $\partial \mathbb{D}$.
\end{prop}
\begin{proof}
We consider the orbits for parameters $\xi$ in a small circular interval $[s,t] \subset \partial \mathbb{D}$, with $s \neq t$.
The initial values $R_0= \xi$ lie in this interval. Since the map
$$
h: R \mapsto \left(\frac{R+b}{b R+1}\right)^d
$$
is expanding, this interval is mapped to an interval $[h(s), h(t)]$ which is strictly larger. Let us write for $\xi\in [s,t]$, $R_1(\xi) = f_{\xi,b}(R_0(\xi))$ and $R_{n+1}(\xi) = f_{\xi,b}(R_n(\xi))$. By Lemma \ref{lemma:covering} the covering map $h$ is orientation preserving. Noting that $f_{\xi,b}(R) = \xi \cdot h(R)$ it follows that the length $\ell[R_1(s),R_1(t)]$ satisfies
$$
\ell[R_1(s),R_1(t)] = \ell[sh(s),th(t)]=\ell[sh(s),sh(t)] + \ell[s,t]
$$
Hence as $n \rightarrow \infty$ we have $\ell[R_n(s),R_n(t)]\rightarrow \infty$, counting multiplicity. Thus there must exist $\xi \in [s,t]$ and $n \in \mathbb N$ for which $f_{\xi,b}^n(\xi) = -1$.
\end{proof}

\begin{remark}\label{rem:dens}
We remark that this proposition combined with Corollary~\ref{cor:tree to dynamics} implies that for $d\geq 2$ and for $b\in (0,\frac{d-1}{d+1}]$ the roots of the partition function of the Ising model for all graphs of maximum degree $d+1$ lie dense in the unit circle. In particular for $b\in (0,1)$, the roots for all graphs lie dense in the unit circle.
\end{remark}

We next look at the case $b\in (b_c,1)$.
\begin{prop}
Let $b_c < b < 1$ and let $\xi_0$ be such that $f_{\xi_0,b}$ has no attracting or parabolic fixed point on $\partial \mathbb D$. Then there are parameters $\xi$ arbitrarily close to $\xi_0$ and $n \in \mathbb N$ for which $f_{\xi,b}^n(\xi) = -1$.
\end{prop}
\begin{proof}
By the assumption that  $f_{\xi_0,b}$ has no attracting or parabolic fixed point on $\partial \mathbb D$, it follows that both $\mathbb D$ and $\overline{\mathbb D}^c$ are attracting basins, and hence the orbits of the two critical points stay bounded away from the Julia set $J = \partial \mathbb D$. It follows that $J$ is a hyperbolic set, i.e. that there exists a metric on $J$, equivalent to the Euclidean metric, with respect to which $f$ is a strict expansion. We will refer to this metric as the hyperbolic metric on $J$.

The proof concludes with an argument similar to the one used in Proposition \ref{expansion1}. For a circular interval $I \subset \partial \mathbb D$ we denote by $\mathrm{length} I$ the diameter with respect to the hyperbolic metric on $\partial \mathbb D$. Let $[s,t] \subset \partial \mathbb D$ be a proper subinterval containing $\xi_0$, small enough so that the maps $f_{\xi,b}$ for $\xi \in [s,t]$ are all strict expansions with respect to the hyperbolic metric obtained for the parameter $\xi_0$. It follows that
$$
\begin{aligned}
\mathrm{length}[f_{s,b}(s), f_{t,b}(t)] = & \mathrm{length}[f_{s,b}(s), f_{s,b}(t)] + \mathrm{length}[f_{s,b}(t),f_{t,b}(t)]\\
> & \kappa \cdot \mathrm{length}[s,t],
\end{aligned}
$$
where the equality follows from the fact that the maps $f_{\xi,b}$ are all orientation preserving, and the constant $\kappa > 1$ is a uniform lower bound on the expansion of the maps $f_{\xi,b}$ for $\xi \in [s,t]$.

By induction it follows that $\mathrm{length}[f_{s,b}^n(s), f_{t,b}^n(t)] > \kappa^n \cdot \mathrm{length}[s,t]$, counting multiplicity. Thus for sufficiently large $n$ the interval $[f_{s,b}^n(s), f_{t,b}^n(t)]$ will contain the unit circle, proving the existence of a parameter $\xi \in [s,t]$ for which $f_{\xi,b}^{n+1}(+1) = -1$.
\end{proof}

Together with Lemma~\ref{lem:derivative 1} and Theorem~\ref{thm:im. at. interval}, this result completes the proof of Theorem~\ref{thm:dynamicsF} (ii).

\subsection{Proof of Theorem~\ref{thm:dynamicsAF}}
\subsubsection{Proof of part (i)}
We start by proving some observations indicating behavior different from the ferro-magnetic case.
\begin{lemma}
Let $b \in (1,1/b_c)$, and let $\xi$ be such that $f_{\xi, b}$ has a parabolic fixed point $R_0$ on $\partial \mathbb D$. Then $\xi$ does not lie in the shortest closed circular interval bounded by $1$ and $R_0$.
\end{lemma}
\begin{proof}
Recall that $R_0$ is not equal to $+1$ or $-1$, and that $|f_{1,b}^\prime|$ is minimal at $+1$ and increases monotonically with $|\mathrm{Arg}(R)|$, and is therefore strictly smaller than $1$ on the open circular interval bounded by $1$ and $R_0$. By integrating $f^\prime_{1,b}$ over this open interval it follows from $f_{1,b}(1) = 1$ that $|\mathrm{Arg}(f_{1,b}(R_0))| < |\mathrm{Arg}(R_0)|$. Since $f_{1,b}$ is orientation reversing it follows that $\mathrm{Arg}(f_{1,b}(R_0))$ and $\mathrm{Arg}(R_0)$ have opposite sign. The statement now follows from
$$
\xi = \frac{R_0}{f_{1,b}(R_0)}.
$$
\end{proof}

\begin{lemma}
Let $b \in (1, 1/b_c)$, and let $\xi$ be such that $f = f_{\xi, b}$ has a parabolic fixed point $R_0$ on $\partial \mathbb D$. The shortest closed circular interval $I$ bounded by $1$ and $\xi$ cannot be forward invariant.
\end{lemma}
\begin{proof}
First observe that the parabolic fixed point is not equal to $1$. By the previous lemma it also cannot be equal to $\xi$. %In the latter case the circular interval bounded by $1$ and $\xi$ could not be forward invariant, since $f|_{\partial \mathbb D}$ is orientation reversing.

Suppose now that $f(I)\subset I$ for the purpose of contradiction. Then the open interval $I^\circ$ is also forward invariant, and since neither $1$ nor $\xi$ is equal to the parabolic fixed point, it follows that $I^\circ$ cannot be contained in the parabolic basin. Hence $I^\circ$ must intersect the Julia set, say in a point $p$. Let $U$ be a sufficiently small open disk centered at $p$ so that $U \cap \partial \mathbb D \subset I^\circ$. Since $p$ lies in the Julia set, it follows that
$$
\bigcup_{n \in \mathbb N} f^n(U)  = \widehat{\mathbb C}.
$$
Since $f$ is forward invariant on $\mathbb D\cup \widehat{\mathbb{C}}\setminus \overline{\mathbb D}$, this contradicts the assumption that $f$ is forward invariant on $I$.

Note that we used here that the exceptional set of the rational function $f$ is empty, which follows immediately from the fact that there are no attracting periodic cycles. We recall that we refer to \cite{Mil} for background on complex dynamical systems.

\end{proof}

It follows from the above lemma that the situation is different from the orientation preserving case: when $f=f_{\xi, b}$ has an attracting fixed point on $\partial \mathbb D$ the point $+1$ does not necessary lie in the immediate attracting interval. If it did, then $f$ would be forward invariant on the shortest interval with boundary points $1$ and $\xi$, which cannot happen for $\xi$ close to the parabolic parameter as follows from the lemma above.
%In fact, when $\xi$ is chosen sufficiently close to a parabolic parameter the point $+1$ \emph{cannot} be in the attracting interval.

Recall that the boundary points of the immediate attracting interval form a repelling periodic cycle. Since for $\xi$ near $+1$ the point $R= +1$ does lie in the immediate attracting interval, while for $\xi$ near the parabolic parameters the point $+1$ does not, it follows by continuity of the repelling periodic orbit that there must be a parameter for which $+1$ is one of the boundary points. In fact, it follows quickly from the fact that $|f^\prime|$ strictly increases with $|\mathrm{Arg}(R)|$ that there exists a unique $\alpha=\alpha_b\in(0,\pi)$ (with $\alpha_b<\theta_b)$ such that for $\xi=e^{\pm i\alpha}$, $+1$ is a boundary point of the immediate attracting interval.
To see that $\alpha$ is unique, suppose there is another such $\alpha'\in (0,\pi)$. We may assume that $\alpha'>\alpha$. Set $\xi=e^{i\alpha}$ and $\xi'=e^{i\alpha'}$.
Then since $f_{\xi',b}(\xi')=\frac{\xi'}{\xi} f_{\xi}(\xi')$ and since $|f_{\xi,b}^\prime(\xi)|>1$ (as $|(f^{\circ 2})^\prime(1)|>1$), it follows that the distance between $f_{\xi',b}(\xi')$ and $f_{\xi,b}(\xi)$ is strictly larger than the distance between $\xi$ and $\xi'$, and hence would need to be a positive multiple of $2\pi$ larger. But then $f_{\xi^\prime,b}$ would map the attracting interval to the entire unit circle, which gives a contradiction.

We note that $\alpha_b$ is the solution in $(0,\pi)$ to
\begin{equation}\label{eq:def alpha}
e^{i\alpha_b}\cdot\left(\frac{e^{i\alpha_b}+b}{e^{i\alpha_b}b+1}\right)^d=1
\end{equation}
with minimal argument. See Figure \ref{fig:lambda01} for the values of $\alpha_b$ and the parabolic parameter $\theta_b$ for varying values of $b > 1$. For a comparable curve depicting the values of $\theta_b$ for $b<1$, see Figure 2 from \cite{Roederetal18}.

\begin{figure}
\includegraphics[width=\textwidth]{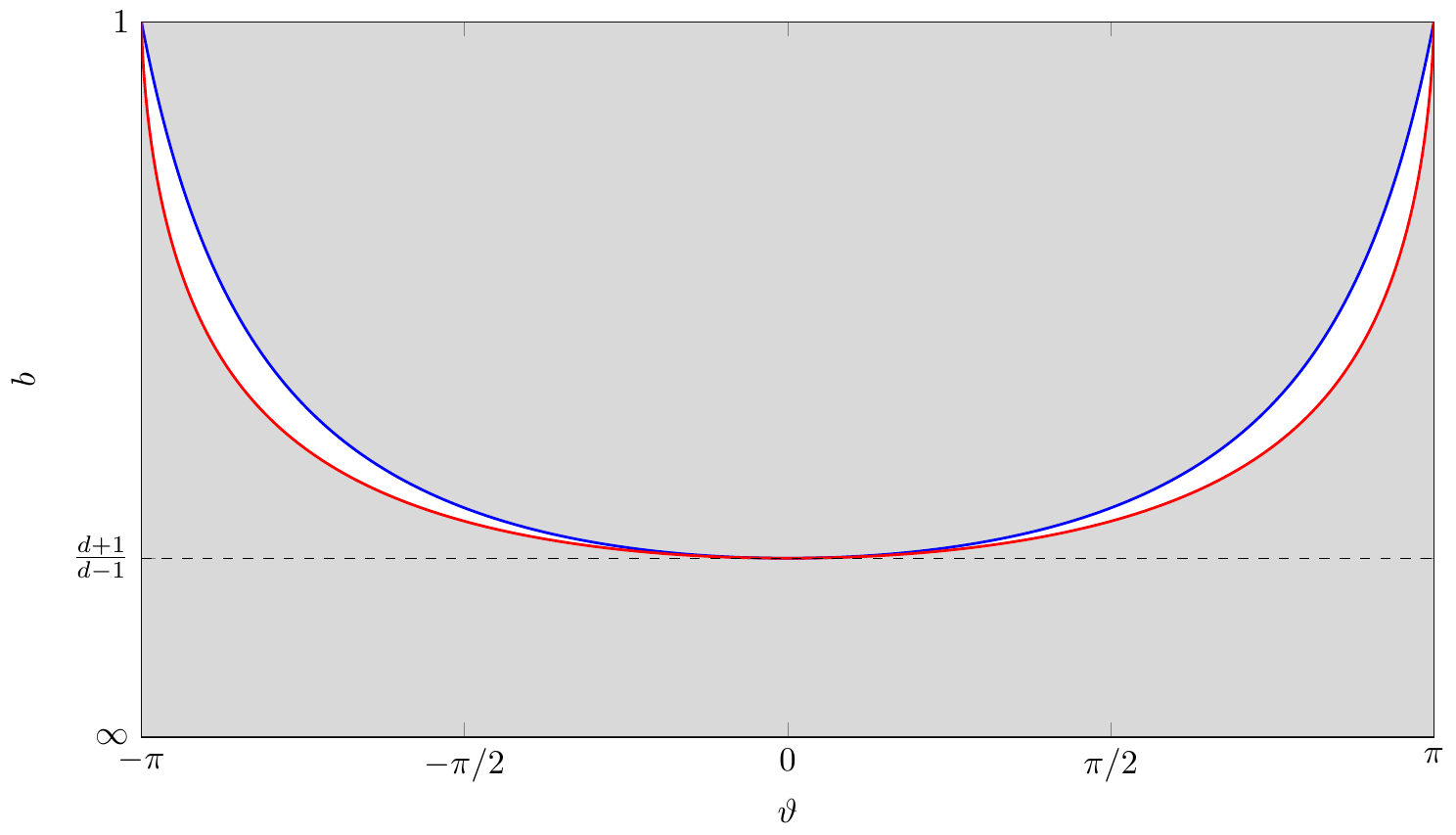}
\caption{The values of $\alpha_b$ (upper curve, in blue) and $\theta_b$ (lower curve, in red) for $d=2$, as $b$ varies from $1$ to $\infty$.}
\label{fig:lambda01}
\end{figure}

We summarize the above discussion in the following theorem
\begin{thm}\label{thm:second im. at. interval}
Let $b\in (1,1/b_c)$ and let $\alpha=\alpha_b$.
If $\vartheta\in (-\alpha,\alpha)$ and $\xi = e^{i\vartheta}$, then the point $+1$ lies in the immediate attracting interval.
%, while $-1$ does not.
\end{thm}
%\begin{proof}
%The fact that $+1$ does lie in the immediate attracting interval follows from the above discussion.
%To see that $-1$ does not lie in the immediate attracting basin, it is sufficient to argue that the imaginary parts of the attracting fixed point $R_0$ and $\xi$ have the same sign, i.e. that the two points either lie both in the upper half circle or in the lower half circle. This is clearly the case when the attracting fixed point lies close to $+1$. The fact that it holds for any attracting fixed point $R_0$ follows from the fact that the map $R_0 \mapsto \xi$ is continuous and injective. See the proof of Theorem~\ref{thm:im. at. interval} for the discussion of this map.
%\end{proof}

To finish the proof of Theorem~\ref{thm:dynamicsAF} (i) we need to show that $I_b$ (which was defined as the shortest closed circular interval with boundary points $1$ and $\xi$) is forward invariant for $f$.
This follows since  $f:I_b\to \partial \mathbb D$ is an orientation reversing injective contraction (with respect to the hyperbolic metric on the attracting basin).

\subsubsection{Proof of part (ii)}
We first note that contrary to the case that $b<1$, one cannot expect parameters $\xi$ with $-1 \in \{f_{\xi, b}^n(+1)\}$ arbitrarily close to any point
$$
\xi_0 \in \partial \mathbb D \setminus \{e^{i \vartheta} : \vartheta \in (-\theta_b, + \theta_b)\},
$$
as there will be $\xi_0$ for which the point $+1$ lies in the attracting basin, just not in the \emph{immediate} attracting basin. In this case the orbit of $\xi_0$ will still converge to the attracting fixed point, and, except for at most countably many parameters $\xi_0$, will avoid $-1$. This is then still the case for $\xi$ sufficiently close to $\xi_0$.

Our goal is to show that there exist parameters $\xi \in \mathbb C$ arbitrarily close to $\xi_0 :=e^{i\alpha_b}$  for which the orbit of $+1$ contains $-1$ (where $\alpha_b$ is defined in \eqref{eq:def alpha}); the complex conjugate $e^{-i\alpha_b}$ is completely analogues. Recall that $f(\xi_0) = 1$, and that this periodic orbit is repelling. It follows that the point $R= +1$ cannot be \emph{passive}, i.e. the family of holomorphic maps
$$
g_n(\xi) = f_{\xi,b}^n(+1)
$$
cannot form a normal family in a neighborhood of $\xi_0$. To see this, note that for an open set of parameters $\xi$ (for example, those for which $+1$ does lie in the immediate attracting interval) accumulating on $\xi_0$ the maps $g_n(\xi)$ converge to $R_0(\xi)$, but the points $g_n(\xi_0)$ remain bounded away from $R_0(\xi_0)$. Thus the sequence of maps $g_n$ cannot have a convergent subsequence in any neighborhood of $\xi_0$.

Recall that the strong version of Montel's Theorem says that a family of holomorphic maps into the Riemann sphere avoiding three distinct points is normal. We claim that this implies that the point $-1$ cannot be avoided for all parameters in a neighborhood of $\xi_0$. Of course, if $-1$ is avoided, then so are all its inverse images. Since the map $f_{\xi_0,b}$ induces a $d$-fold covering on the unit circle, it is clear that there exist points $z_{-1}, z_{-2}$, distinct from each other as well as from $-1$, such that $f_{\xi_0,b}(z_{-2}) = z_{-1}$ and $f_{\xi_0,b}(z_{-1}) = -1$. Since $f_{\xi,b}$ varies holomorphically with $\xi$, it follows that for all parameters $\xi$ sufficiently close to $\xi_0$ we can similarly find points $z_{-2}(\xi)$ and $z_{-1}(\xi)$, varying holomorphically with $\xi$, for which $f_{\xi,b}(z_{-2}(\xi)) = z_{-1}(\xi)$ and $f_{\xi,b}(z_{-1}(\xi)) = -1$.

There is a uniquely defined M\"obius transformation $\varphi_\xi$ (depending holomorphically) on $\xi$ that maps $-1$, $z_{-1}(\xi)$ and $z_{-2}(\xi)$ to $-1$, $z_{-1}$ and $z_{-2}$ respectively. Define the holomorphic maps
$$
h_n(\xi) = \varphi_\xi \circ g_n(\xi).
$$
Since the family $\{g_n\}$ cannot be normal near $\xi_0$, neither can the family $\{h_n\}$. Hence the latter cannot avoid the three distinct points $\{-1, z_{-1}, z_{-2}\}$, which implies that there are $\xi$ arbitrarily close to $\xi_0$ for which there exist $n \in \mathbb N$ such that $g_n(\xi)$ lies in $\{-1, z_{-1}(\xi), z_{-2}(\xi)\}$. This completes the proof.

\section{Zero-free regions}\label{sec:zero-free}
It is our aim in this section to prove part (i) of Theorems A and B.
To this end let us fix $d\in \mathbb{N}_{\geq 2}$ and $b\in (\frac{d-1}{d+1},1)\cup(1,\frac{d+1}{d-1})$ and $\vartheta\in (-\theta_b,\theta_b)$ when $b<1$ and $\vartheta\in (\alpha_b,\alpha_b)$ when $b>1$.
Let us fix $\xi=e^{i\vartheta}$, and write $f=f_{\xi,b}$.

Our strategy is as follows.
First we give a forward invariant domain for the function $f$ in Subsection~\ref{ssec:domain} and then show that this domain is invariant for a multivariate version of the function $f$.
We then use this to prove part (i) for trees with boundary conditions in Subsection~\ref{ssec:trees with b}.
Finally, we prove the result for all graphs in Subsection~\ref{ssec:graphs}

\subsection{Invariant domain}\label{ssec:domain}
We let $I_b$ be the forward invariant interval for $f$ from Theorem~\ref{thm:dynamicsF} (i) when $b<1$ and from Theorem~\ref{thm:dynamicsAF} (i) when $b>1$.

We introduce the $\mathbb R_{\ge 0}$-cone generated by $I_b$:
\begin{equation}\label{eq:cone}
C=C_b:=\{z=re^{i\phi}\mid 0 \le r \le \infty , \phi \in I_b\}.
\end{equation}
Recall the M\"obius transformation
$$
g: R \mapsto \frac{R + b}{b R + 1}.
$$

\begin{lemma}\label{lemma:forward}
The cone $C$ is forward invariant under $f$, that is, for any $z\in C$, $f(z)\in C$.
\end{lemma}
\begin{proof}
We first consider the case $b<1$, where $f$ is orientation preserving on the circle.
It suffices to show that the half lines $H$ and $\mathbb{R}_{\geq 0}$ bounding $C$ are mapped into $C$ by $f$.
Since $\mathbb{R}_{\geq 0}$ is mapped to the half line through $\xi$ and $\xi\in I_b$, it remains to show this for $H$.
%Recall the M\"obius transformation $z\mapsto g(z)=\frac{z+b}{b z+1}$.
%Then $f(z)=\xi g(z)^d$.
%We may assume that $\phi\geq 0$.

We claim that $g(e^{i\phi_0})$ equals the principal value of $(e^{i\phi_0}\xi^{-1})^{1/d}$.
%(We note that since $f(1)=\xi$
To see this, let us denote the preimages of $e^{i\phi_0}$ under $f$ as $R_0=e^{i\phi_0},\ldots,R_{d-1}$.
They are exactly equal to the $d$ values $g^{-1}((e^{i\phi_0}\xi^{-1})^{1/d})$.
Since $f$ is forward invariant on $I_b$ and orientation preserving on $\partial \mathbb D$, none of the $R_1,\ldots,R_{d-1}$ lie in the interval $I_b$.
Since $g$ preserves orientation and since $g(1)=1$, it follows that $g(e^{i\phi_0})$ is indeed equal to the principal value of $(e^{i\phi_0}\xi^{-1})^{1/d}$. We will write $z=e^{i\phi}:=g(e^{i\phi_0})$ from now on.

Now we consider the circle $\gamma$ through the three points $b,1/b $ and $z$.
This circle is exactly the image of the line $L:=H\cup ( -H)$ under the transformation $g$, as $0\mapsto b$, $\infty\mapsto 1/b$ and $e^{i\phi_0}\mapsto z$ under $g$. %The image of $H$ is equal to the part of the circle with nonnegative imaginary part.

We next claim that the half line $z \cdot \mathbb R_{\ge 0}$
only intersects $\gamma$ in the point $z$. Indeed, this follows from the fact that the line $H$ intersects $\partial \mathbb D$ normally, and $g$ is conformal. Hence the circle $\gamma$ intersects the unit circle $\partial \mathbb D$ normally in $z$, which implies that
the line $z \cdot \mathbb R$ is the tangent line to $\gamma$ at $z$. In other words, the argument of $g(z)$ for $z \in C$ is extremal when $z = R_0$.

This then implies that for any point $y$ on the image of $H$ under $g$ we have that the argument of $\xi y^d$ is between $0$ and $\phi_0$.
In other words, the image of $H$ under $f$ is contained in $C$, as desired.

Now suppose that $b>1$. Again it suffices to show that the two half lines $H:=\{r\xi\mid r\geq 0\}$ and $\mathbb{R}_{\geq 0}$ are mapped into $C$ by $f$, and again we only need to check this for the half line $H$. Its image under $g$ is contained in the circle through $b, \frac{1}{b}$ and $g(\xi)$, which again intersects $\partial \mathbb D$ normally in $g(\xi)$. In this case the argument of $g(z)$ is therefore extremal when $z = \xi$. It follows that the half line $g(\xi)\cdot \mathbb R_{\ge 0}$ intersects $g(H)$ only in $g(\xi)$, and the proof is essentially the same as for the $b<1$ case.
\end{proof}

\begin{figure}%\label{fig:plaatje}
\centering%
\begin{subfigure}[b]{0.45\textwidth}
\centering\includegraphics[angle=270,origin=c,width=\textwidth]{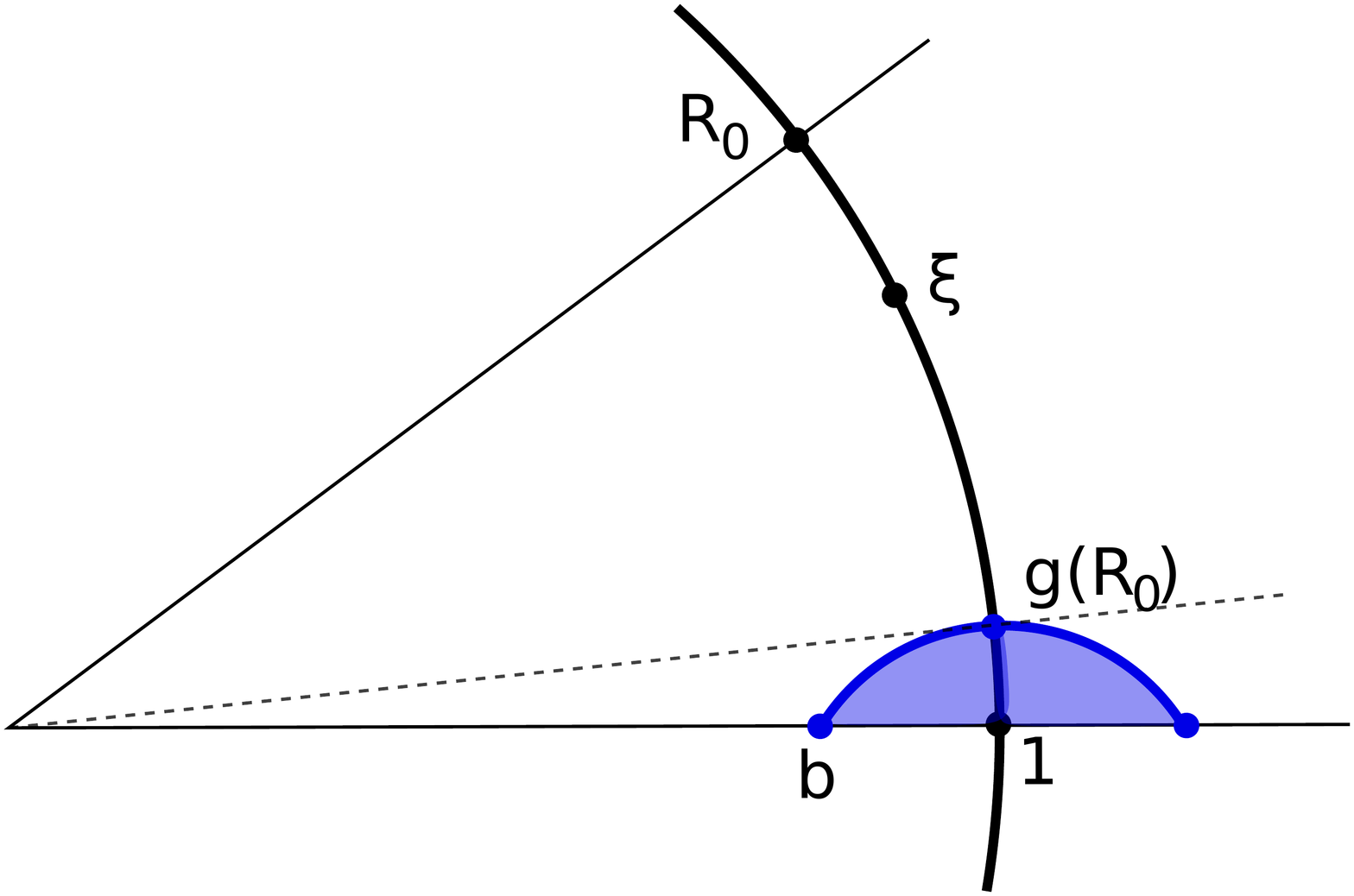}
            %\caption*{$b<1$}
            \label{eersteplaatje}
\end{subfigure}%
\quad     %add desired spacing between images, e. g. ~, \quad, \qquad etc.
\begin{subfigure}[b]{0.45\textwidth}
            \centering
\includegraphics[angle=270,origin=c,width=\textwidth]{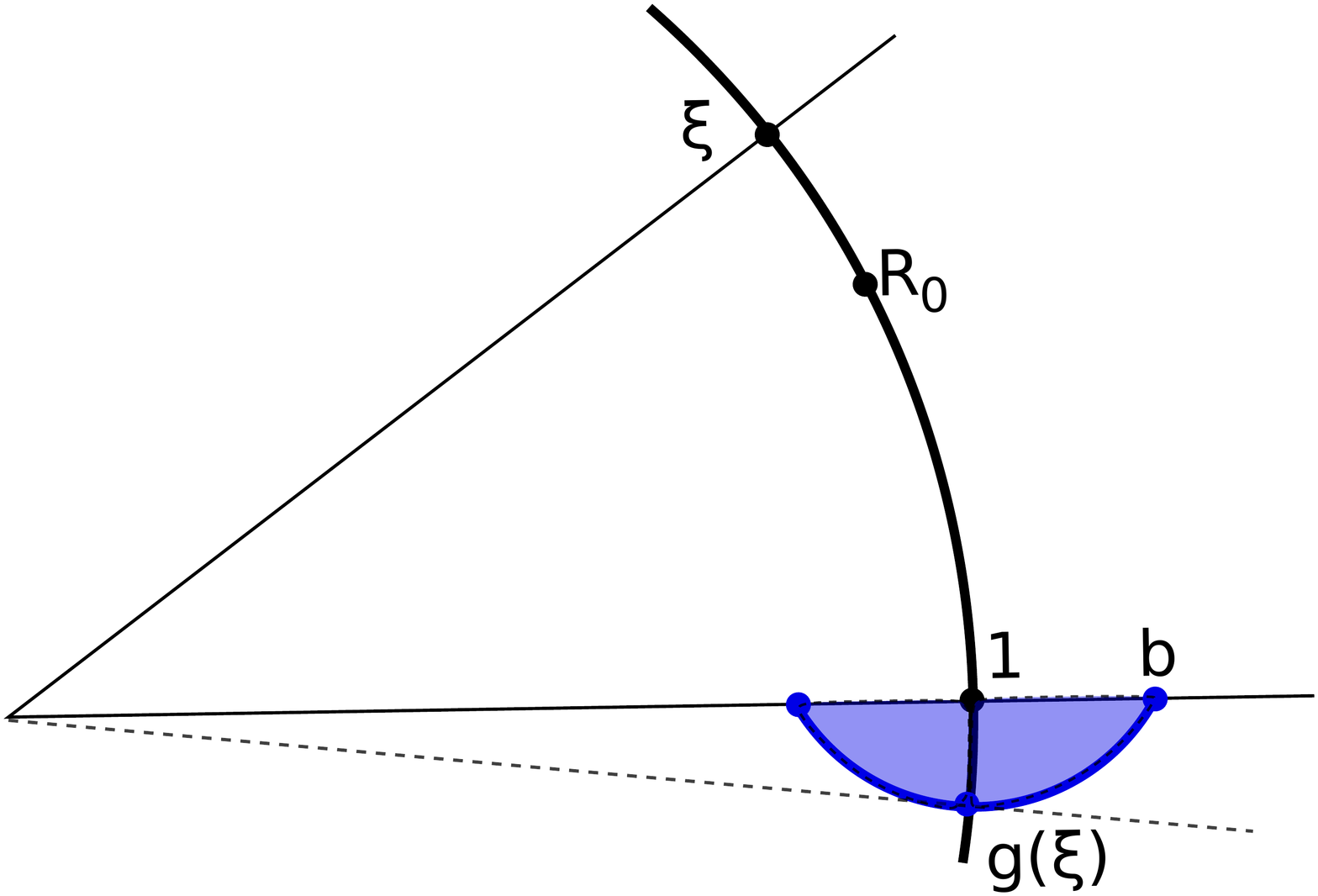}
            %\caption*{$b>1$}
            \label{tweedeplaatje}
\end{subfigure}
\caption{The image $g(C)$ for $b<1$ on the left and for $b>1$ on the right.}
    \label{plaatjes}
\end{figure}

\begin{lemma}
For each $R_1,R\in I_b$ we have $g(R_1)R\neq -1$.
\end{lemma}
\begin{proof}
Let us first consider the case $b< 1$. We claim that it suffices to prove the statement for $R_1=R$.
Indeed, suppose there exists $R_1,R\in I_b$ such that $g(R_1)R=-1$. Let
\[m=\min\{|R_1-R|\mid g(R_1)R=-1, R_1,R\in I_b\}.\]
If $m\neq 0$, then take $R_1,R\in I_b$ such that $g(R_1)R=-1$ and such that $|R_1-R|=m$.
We may assume $\mathrm{arg}(R_1)>\mathrm{arg}(R)$.
Then there is $R'\in I_b$ with $\mathrm{arg}(R')>\mathrm{arg}(R)$ and, using that $g$ is orientation preserving, $\mathrm{arg}(R'_1)<\mathrm{arg}(R_1)$ with $g(R'_1)R'=-1$. A contradiction.

The condition that $g(R)R=-1$ translates to,
$$
R^2+2b R+1=0.
$$
It can easily be checked that the coefficient of $R$, $2b$, is strictly larger than the coefficient of $R$ found in equation \eqref{eq:beta<1}, as
$$
2b^2 > d(b^2 - 1) + (1+b^2).
$$
Thus, the solutions of $g(R)R=-1$ lie closer to $-1$ than $R_0$, and there is no such solution in $I_b$.

\medskip

When $b > 1$ the maps $f$ and $g$ are orientation reversing. Since $f$ maps $I_b$ into itself, it follows that $g^d = \xi^{-1} \cdot f$ maps $I_b$ into the circular interval between $1$ and $\xi^{-1}$. Since $g(1) = 1$ it follows that $g$ maps the interval $I_b$ into the circular interval bounded by $1$ and the principal value of $\xi^{-\frac{1}{d}}$.
Hence $|\mathrm{Arg}(R_1g(R)) |$ is maximal when $R_1 = \xi$ and $R = 1$, in which case $R_1g(R) = \xi \neq -1$, which completes the proof.
\end{proof}

Let us consider for $\mu,b\in \mathbb C$ the map $F_{\mu,b}:\hat{\mathbb{C}}^d\to \hat{\mathbb{C}}$ defined by
\[
F_{\mu,b}(R_1,\ldots,R_d)\mapsto \mu\cdot \prod_{i=1}^d \frac{R_i+b}{b R_i+1}.
\]
\begin{prop}\label{prop:stable region}
%Let $b\in (\frac{d-1}{d+1},1)\cup (1,\frac{d+1}{d-1})$ and let $\theta=\theta_b$ be defined as in Theorem \ref{thm:dynamics}.
%Fix any $\vartheta\in (-\theta,\theta)$ and set $\xi=e^{i\vartheta}$. Then f
For any $R_1,\ldots,R_{d+1}\in C=C_b$ and any $r\geq 0$:
\begin{itemize}
\item[(i)] $F_{r\cdot \xi,b}(R_1,\ldots,R_d)\in C_b$;
\item[(ii)] $|\mathrm{Arg}(g(R_{d+1})F_{r\cdot \xi,b}(R_1,\ldots,R_d))|<\pi.$
\end{itemize}
\end{prop}
\begin{proof}
Since $C$ is a cone, it suffices to prove this for $r=1$.
Let us write $F=F_{\xi,b}$.
First we prove (i)
We note that $\log(C)$ is a rectangle and in particularly it is convex. (We take the branch of the logarithm that is real on the positive real line.)
Let us fix $R_1,\ldots,R_d\in C$. Then, since $f(R_i)\in C$ by Lemma \ref{lemma:forward}, we know that $\log(f(R_i))\in \log(C)$.
Then,
\[
\log(F(R_1,\ldots,R_d))=\sum_{i=1}^d \frac{\log(\xi)}{d}+ \log\left (\frac{R_i+b}{b R_i+1}\right)=\sum_{i=1}^d \frac{1}{d} \log(f(R_i)),
\]
and this is contained in $\log(C)$ by convexity.
This implies that $F(R_1,\ldots,R_d)$ is contained in $C$, as desired.

To prove (ii), by (i) it suffices to show that for any $R_1,R\in C$ we have $|\mathrm{Arg}(g(R_1)R)|<\pi$.
Suppose to the contrary that for some $R_1\in C$ and $R\in C$ we have $g(R_1)R\in \mathbb{R}_{<0}$.
The positive half line through $R_1$ intersects the unit circle normally at some point $R_1'$.
It follows that $|\text{Arg}(g(R'_1))|$ is not smaller than $|\text{Arg}(g(R_1))|$, cf. Figure~\ref{plaatjes}.
Since $C$ is a cone, it then follows that we can find $R_1'',R'\in C$ of norm one such that $g(R''_1)R'=-1$.
The previous lemma however gives that $g(R''_1)R'\neq -1$, a contradiction.
\end{proof}

\begin{remark}\label{rem:zero}
For the purpose of finding an efficient algorithm for approximating $Z_G(\mu,b)$, for any graph $G$ of maximum degree at most $d+1$, it is important that $Z_G$ does not vanish for $\mu$ sufficiently close to zero. By the result of Lee-Yang this is immediate for $0 < b < 1$. For $b >1$ the statement follows quickly from the formula for $F_{\xi,b}(R_1, \ldots, R_d)$. Indeed, write $D_r = \{z: |z|< r\}$ for some $0 < r < \frac{1}{b}$ and consider $R_1, \ldots, R_d \in D_r \cup \{\infty\}$. Then one observes that the possible values of
$$
\prod_{i = 1}^d \frac{R_i + b}{b R_i + 1}
$$
are bounded. Therefore for $|\mu|$ sufficiently small one obtains that
$$
F_{\mu,b}(R_1, \ldots, R_d) \in D_r.
$$
In the analysis that follows the forward invariant set $D_r \cup \{\infty\}$, which does not contain the point $-1$, can therefore play the same role as the cone $C_b$ studied in Proposition~\ref{prop:stable region}.

We note that the zero-free neighborhood of the origin can also be derived by using a result of Ruelle \cite{Ru71}. In fact, this argument gives a neighborhood that is independent of the maximum degree of the graph.
\end{remark}

\subsection{Trees with boundary conditions}\label{ssec:trees with b}
Given a graph $G=(V,E)$. Let $X\subseteq V$ and let $\tau:X\to \{0,1\}$ be a boundary condition on $X$.
We call vertices $u\in X$ \emph{fixed} and vertices $v\in V\setminus X$ \emph{free}.

\begin{prop}\label{prop:good ratios trees}
%Let $d\geq 2$. Let $b\in (\frac{d-1}{d+1},1)\cup (1,\frac{d+1}{d-1})$, let $\theta$ be the constant from Theorem~\ref{thm:dynamics}, and let $\xi=r\cdot e^{i \vartheta}$ for $\vartheta\in (-\theta,\theta)$ and $r>0$.
Let $G=(V,E)$ be any tree in $\mathcal{G}_{d+1}$.
Let $X\subseteq V$ and let $\tau:X\to \{0,1\}$ be a boundary condition on $X$.
Set $\xi_v=\xi$ for $v\notin X$ and choose $\xi_v\neq 0$ arbitrarily for $v\in X$.
Then for any $v\in V$ the ratio $R_{G,\tau,v}$ does not lie in $\mathbb{R}_{<0}$. %contained in the cone $C:=C_b$ as defined in~\eqref{eq:cone}.

\end{prop}
\begin{proof}
Let $G=(V,E)\in \mathcal{G}_{d+1}$ and let $v\in V$.
We start by proving the following statements assuming that the degree of $v$ is at most $d$:
\begin{itemize}
\item[(i)]$Z_{G,\tau_{v,0}}\neq 0$, or $Z_{G,\tau_{v,1}}\neq 0$, and
\item[(ii)] the ratio $R_{G,\tau,v}$ is contained in the cone $C=C_b$.
\end{itemize}
As $-1\notin C$, this is clearly sufficient for this case.

The proof is by induction on the number of vertices of $G$.
If this number is equal to $1$, then either $X=V$, or $X=\emptyset$.
In the first case, either we have $Z_{G,\tau_{v,0}}\neq 0$, or $Z_{G,\tau_{v,1}}\neq 0$, as desired.
Similarly, we have $R_{G,\tau,v}$ is either equal to $0$, or to $\infty$.
In the latter case, we have $Z_{G,\tau_{v,0}}=1$, and $Z_{G,\tau_{v,1}}=\xi$ and hence $R_{G,\tau,v}=\xi\in C$.
This verifies the base case.

Let us assume that number of vertices is at least $2$.
Let $v_1,\ldots,v_m$ be the neighbors of $v$ and let $G_1,\ldots,G_m$ be the components of $G-v$ containing $v_1,\ldots,v_m$ respectively.
We just write $\tau$ for the restriction of $\tau$ to $X\cap V(G_i)$.
For $j=0,1$ we write $\tau_{i,j}$ for the boundary condition on $(X\cup \{v_i\})\cap V(G_i)$ obtained from $\tau$ where $v_i$ is set to $j$.
Suppose first that $v\notin X$ or that $v\in X$ and $\tau(v)=0$.
Then
\begin{equation}\label{eq:Z_0 not 0}
Z_{G,\tau_{v,0}}=\prod_{i=1}^m (b Z_{G_i,\tau_{i,1}}+Z_{G_i,\tau_{i,0}}).
\end{equation}
Now by induction, since the number of vertices in each $G_i$ is less than that in $G$, we know that for each $i$, $Z_{G_i,\tau_{i,1}}\neq 0$, or $Z_{G_i,\tau_{i,0}}\neq 0$.
Moreover, since for each $i$, $R_i:=R_{G_i,\tau,v_i}\in C$, we have, since $b>0$, $b R_i\in C$ and hence $b Z_{G_i,\tau_{i,1}}+Z_{G_i,\tau_{i,0}}\neq 0$, as $-1\notin C$.
This implies that $Z_{G,\tau_{v,0}}\neq 0$.
One shows with a similar argument (using that $1/b \cdot C=C$) that if $v\in X$ and $\tau(v)=1$, then $Z_{G,\tau_{v,1}}\neq 0$.

To see (ii), if $v\in X$, we have $R_{G,\tau,v}\in \{0,\infty\}$, otherwise, by Lemma~\ref{lem:ratio} we have,
\[
R_{G,\tau,v}=\xi_v\prod_{i=1}^d\frac{R_i+b}{b R_{i}+1}=F_{\xi,b}(R_1,\ldots,R_d),
\]
where in case $m<d$, we set $R_{m+1},\ldots,R_d$ all equal to $1$.
By part (i) of Proposition~\ref{prop:stable region} we conclude that $R_{G,\tau,v}$ is contained in $C$.
This concludes the proof of (i) and (ii).

To finish the proof, we must finally consider the case where the degree of $v$ is equal to $d+1$.
We only need to argue that $R_{G,\tau,v}\notin \mathbb{R}_{<0}$, since the argument that $Z_{G,\tau_{v,0}}\neq 0$, or $Z_{G,\tau_{v,1}}\neq 0$ is the same as above.
We again let $v_1,\ldots,v_{d+1}$ be the neighbors of $v$ and let $G_1,\ldots,G_{d+1}$ be the components of $G-v$ containing $v_1,\ldots,v_{d+1}$ respectively.
Let us write $R_i:=R_{G_i,\tau,v_i}$ for each $i$.
Then, by the first part of the proof we know that for each $u_i$ we have $R_i\in C$, from which we conclude by Lemma~\ref{lem:ratio} and Proposition~\ref{prop:stable region} (ii).
\[
R_{G,\tau,v}=g(R_{d+1})F_{\xi,b}(R_1,\ldots,R_d)\notin \mathbb{R}_{<0},
\]
This concludes the proof.
\end{proof}

\subsection{General bounded degree graphs}\label{ssec:graphs}
Here we conclude the proof of part (i) of Theorems A and B by utilizing Proposition~\ref{prop:good ratios trees}. To do so we need the  \emph{self avoiding walk tree} as introduced by Weitz~\cite{W06}.

Let $G$ be a connected graph of maximum degree at most $d+1$ and fix a vertex $v$ of $G$, which we call the \emph{base vertex}.
The \emph{self avoiding walk tree of $G$ at $v$} is a tree $T=T_{\text{SAW}}(G,v)$ whose vertices consists of walks in $G$ starting at $v$ that are of the form $w=(v,v_1,\ldots,v_n)$ such that the walk $(v,\ldots,v_{n-1})$ is self-avoiding (i.e. a path in the graph theoretic sense). The walks $w$ that are not self-avoiding, that is, whose last vertex closes a cycle in $G$, will be leaves of the tree. Note that walks ending in a leaf of $G$ are automatically leaves of the tree as well.
Two vertices (walks in $G$) $w_1$ and $w_2$ of $T$ are connected by an edge if one is the one-point extension of the other (as walks in $G$).
Note that the maximum degree of $T$ is at most $d+1$.
See Figure \ref{fig:SAW} for an example.

\begin{figure}
\includegraphics[width=\textwidth]{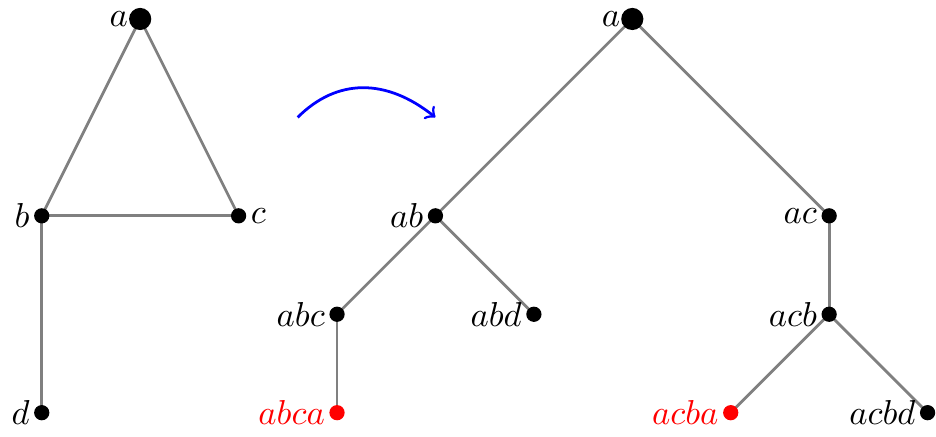}
\caption{A graph and its associated tree of self-avoiding walks. The cycles inducing boundary conditions are marked in red.}
\label{fig:SAW}
\end{figure}

We next fix a boundary condition $\tau_G$ on some of the leaves of $T$.
To do so we fix for each vertex $u$ of $G$ an arbitrary ordering of the edges incident with it.
We only fix leaves corresponding to walks closing a cycle in $G$.
Such a walk will be set to $0$ if the edge closing the cycle is larger than the edge starting the cycle and set to $1$ otherwise.
If $\sigma$ is a boundary condition on a subset of the leaves of the graph $G$, we extend $\tau_G$ by assigning the same value to any vertex of $T$, corresponding to a path ending at such a leaf.
Let $(\xi_u)_{u\in V}$ be complex numbers associated with the vertices of $G$.
We associate these variables to the vertices of $T$ as follows.
For a vertex $w$ of $T$ let $u$ be the last vertex in the corresponding walk in $G$. Then we set $\xi_w:=\xi_u$.

\begin{prop}\label{prop:SAW}
Let $G=(V,E)$ be a connected graph of maximum degree at most $d+1$, and with vertex $v\in V$.
Let $X\subset V\setminus \{v\}$ be a collection of leaves of $G$, and let $\sigma:X\to\{0,1\}$ be a boundary condition on $X$.
Set $\xi_u= \xi$ for all $u\notin X$, and choose $\xi_u \neq 0$ arbitrarily for $u \in X$.
Then
\begin{equation}\label{eq:equal ratios}
R_{T,\tau_G,v}=R_{G,\sigma,v},
\end{equation}
where both sides are considered as rational functions in $\xi$.
\end{prop}

We remark that when all $\xi_u$ are positive this lemma is essentially due to Weitz~\cite{W06} (even though in~\cite{W06} only the independence polynomial was considered).
%Since we work with complex valued $\xi$ we need to be careful and the proof of the lemma becomes a little subtle.

\begin{proof}
We use induction on the number of free vertices of $G$. If the number of free vertices is $1$ then the free vertex is $v$, and all other vertices are leaves. Thus $G$ is a tree and equals its tree of self-avoiding walks, and the statement is immediate.

We may therefore assume that $V$ has at least $2$ free vertices, and that equality \eqref{eq:equal ratios} holds for graphs with fewer free vertices. Let us denote the neighbors of $v$ by $u_1, \ldots , u_m$. We construct a graph $\hat{G}$ by replacing the vertex $v$ with $v_1, \ldots , v_m$, each $v_j$ having exactly one neighbor: the vertex $u_j$. To each of the vertices $v_j$ we assign the external field parameter $\xi^{\frac{1}{m}}$, using the same holomorphic branch of the $m$-th root for all the vertices $v_j$. We introduce boundary conditions $\sigma_i$, for $i = 0,\ldots m$, each extensions of $\sigma$, by setting $\sigma_i(v_j) = 1$ when $j \le i$ and $\sigma_i(v_j) = 0$ when $j > i$. Thus $\sigma_0$ assigns $0$ to all the vertices $v_j$, and $\sigma_m$ assigns $1$ to all the vertices $v_j$. It follows from our choice of the external field parameter $\xi^{\frac{1}{m}}$ that
$$
Z_{\hat{G},\sigma_0} = Z_{G, \sigma_{v,0}},
$$
and
$$
Z_{\hat{G},\sigma_m} = Z_{G, \sigma_{v,1}},
$$
and therefore
$$
R_{G,\sigma,v} = \frac{Z_{\hat{G},\sigma_m}}{Z_{\hat{G},\sigma_0}} = \prod_{i=1}^{m} \frac{Z_{\hat{G},\sigma_i}}{Z_{\hat{G},\sigma_{i-1}}}.
$$
Writing $\hat{\sigma}_i$ for the restriction of $\sigma_i$ obtained by freeing the vertex $v_i$, it follows that
$$
\frac{Z_{\hat{G},\sigma_i}}{Z_{\hat{G},\sigma_{i-1}}} = R_{\hat{G}, \hat{\sigma}_i, v_i}.
$$
Write $\hat{G}_i$ for the connected component of $\hat{G}$ that contains the vertex $v_i$, and observe that $\hat{G}_i$ with boundary condition $\hat{\sigma}_i$ has at most as many free vertices as $G$. Let us stress that $G$ is not necessarily a tree, hence $\hat{G}_i$ can equal $\hat{G}_j$ for $i \neq j$. It therefore follows from our induction hypothesis that
$$
R_{\hat{G}_i - v_i, \hat{\sigma}_i,u_i} = R_{T_{\text{SAW}}(\hat{G}_i - v_i, u_i), \hat{\tau}_i, u_i},
$$
where $\hat{\tau}_i$ is the boundary condition on $T_{\text{SAW}}(\hat{G}_i - v_i, u_i)$ induced by $\hat{\sigma}_i$.
Observing that the vertex $v_i$ has only one neighbor in $\hat{G}_i$, namely $u_i$, it follows from the same argument as used in the proof of Lemma~\ref{lem:ratio} that
$$
R_{\hat{G}_i,\hat{\sigma}_i,v_i} = \xi^{1/d} g(R_{\hat{G}_i- v_i, \hat{\sigma}_i,u_i}) = \xi^{1/d} g(R_{T_{\text{SAW}}(\hat{G}_i- v_i, u_i), \hat{\tau}_i, u_i})
%& = R_{T_{\text{SAW}}(\hat{G}_i, v_i), \hat{\tau}_i, v_i},
$$
where $g$ denotes the M\"obius transformation $R \mapsto \frac{R+b}{bR+1}$. By applying Lemma~\ref{lem:ratio} to $T_{\text{SAW}}(G,v)$, we obtain
$$
R_{G,\sigma,v} = \prod_{i=1}^{m} R_{\hat{G}_i,\hat{\sigma}_i,v_i} = \xi \cdot \prod_{i=1}^{m} g(R_{T_{\text{SAW}}(\hat{G}_i- v_i, u_i), \hat{\tau}_i, u_i}) = R_{T_{\text{SAW}}(G,v),\tau,v},
$$
where the boundary condition $\tau$ on $T_{\text{SAW}}(G,v)$ is obtained from the boundary conditions $\hat{\tau_i}$, and therefore satisfies the following:
\begin{enumerate}
\item[$\bullet$] Walks ending in a leaf $u \in X$ are assigned the boundary condition $\sigma(u)$.
\item[$\bullet$] Walks ending in a leaf $u \notin X$ are not assigned a boundary condition.
\item[$\bullet$] The boundary condition of a cycle $(v, u_i, \ldots, u_j, v)$ depends on the relative ordering of the neighbors $u_i, u_j$ in the arbitrarily chosen numbering $u_1, \ldots , u_m$. We stress that this numbering is identical for all the cycles.
\item[$\bullet$] The boundary condition of a walk that is not a cycle but ends with a cycle is determined in the induction process, again depending on the chosen numbering of the edges incident to the vertex at the start and end of the cycle.
\end{enumerate}
Thus, the boundary condition $\tau$ satisfies the rules described for $\tau_G$, and the proof is complete.
\end{proof}

We will now prove that it was correct to consider the ratios $R$ as rational functions in $\xi$. We remark that in the previous proposition the choice of $\xi$ was irrelevant, while in what follows it plays an essential role in the proof.

\begin{lemma}\label{lemma27}
Under the hypotheses of the previous proposition we have that $Z_{G,\sigma_{v,0}}\neq 0$, and $Z_{G,\sigma_{v,1}}\neq 0$.
\end{lemma}
\begin{proof}
Again we prove the statement by induction on the number of free vertices. When the number of free vertices is $1$ the free vertex is $v$. It follows that
$$
Z_{G, \sigma_{v,0}} = \prod_{\sigma(u) = 1} b\xi_u \neq 0,
$$
and similarly $Z_{G,\sigma_{v,1}} \neq 0$.
%Denoting the neighbors of $v$ by $u_1, \ldots , u_m$ we obtain $Z_{G,\sigma_{v,0}} = \prod_{i=1}^m r_i$, where $r_i$ is either $\xi_{u_i} b$ or $1$, depending on $\sigma(u_i)$. Since $\xi_{u_i}$ was chosen in $\mathbb C \setminus \{0\}$, it follows that  $Z_{G,\sigma_{v,0}} \neq 0$, and similarly $Z_{G,\sigma_{v,1}} \neq 0$.

So let us now assume that $|V\setminus X|>1$. We again denote by $\hat{G}$ the graph obtained by replacing the vertex $v$ by vertices $v_1, \ldots , v_m$, each $v_i$ neighboring only the vertex $u_i$, using again $\xi^{1/m}$ for all the vertices $v_i$. We denote by $\sigma_0, \ldots, \sigma_m$ the extensions of $\sigma$ introduced in the proof of the previous proposition, and we write $\hat{\sigma} = \sigma_i$ for some $i \in \{0, \ldots , m\}$. We will prove that $Z_{\hat{G}, \hat{\sigma}} \neq 0$ for all $i = 0, \ldots , m$, which implies the statement of the lemma since $Z_{G,\sigma_{v,0}} = Z_{\hat{G}, \sigma_0}$ and $Z_{G, \sigma_{v,1}} = Z_{\hat{G}, \sigma_m}$.

Denote by $H$ a connected component of $\hat{G}$. It suffices to show that $Z_{H,\hat \sigma}\neq 0$, as the partition function is multiplicative over components. As $G$ was assumed to be connected, $H$ contains a vertex $v_i\in \{v_1,\ldots,v_m\}$.
Let $H':=H-v_i$. Then $H'$ contains fewer free vertices than $G$.
Let us first assume that $u_i$ is not a fixed leaf in $G$, i.e., either a leaf that is not fixed or not a leaf. Then by the induction hypothesis it follows that $Z_{H',\hat \sigma_{u_i,0}}\neq 0$ and $Z_{H',\hat \sigma_{u_i,1}}\neq 0$, hence by Proposition \ref{prop:SAW} it follows that
\[
R_{H',\hat \sigma,u_i}=R_{T_\text{SAW}(H',u_i),\tau_{H'},u_i},
\]
where by Proposition \ref{prop:good ratios trees} the latter ratio does not lie in $\mathbb{R}_{<0}$.
Since, we have that $Z_{H,\hat \sigma}$ equals either $Z_{H',\hat \sigma_{u_i,0}}+bZ_{H',\hat \sigma_{u_i,1}}$, or $\xi_{v_{i}}(bZ_{H',\hat \sigma_{u_i,0}}+Z_{H',\hat \sigma_{u_i,1}})$, it follows that $Z_{H,\hat \sigma}\neq 0$.

If instead $u_i$ is a fixed leaf in $G$, then $H'=H-v_i$ just consists of the vertex $u_i$ and therefore  $Z_{H',\hat \sigma_{u_i,0}}\neq 0$ and $Z_{H',\hat \sigma_{u_i,1}}=0$, or vice versa, from which it again follows that $Z_{H,\hat \sigma}\neq 0$. This completes the proof.
\end{proof}

\begin{proof}[Proof of part (i) of Theorems A and B]
We may assume that $G$ is connected since the partition function is multiplicative over components.
Fix any vertex $v\in V$, and denote by $\tau$ the empty boundary condition on $G$. % and let for $i=0,1$, $\tau_i$ denote the boundary condition on $\{v\}$ where $v$ is set to $i$.
By the previous lemma both $Z_{G,\tau_{v,0}}$ and $Z_{G,\tau_{v,1}}$ are nonzero.
Moreover, the ratio $R_{G,v}$ is not equal to $-1$ by Propositions~\ref{prop:good ratios trees} and \ref{prop:SAW}.
Therefore by \eqref{eq:important observation} we conclude that $Z_G\neq 0$.
\end{proof}

\section*{Acknowledgements}
We thank Sander Bet, Ferenc Bencs, Pjotr Buys, David de Boer and Eoin Hurley for useful comments on a previous version of this paper.
We moreover thank the anonymous referees for helpful comments improving the presentation of the paper as well as for suggesting a number of important references.

\end{document}